\documentclass[graybox]{svmult}

\usepackage{mathptmx}       
\usepackage{helvet}         
\usepackage{courier}        
\usepackage{type1cm}        
\usepackage{makeidx}         
\usepackage{graphicx}        
\usepackage{multicol}        
\usepackage[bottom]{footmisc}

\newtheorem{defi}{Definition}
\newtheorem{alg}{Algorithm}

\newtheorem{thm}{Theorem}
\newtheorem{lem}{Lemma}
\newtheorem{cor}{Corrolary}

\makeindex             

\begin{document}

\title*{An Algorithm for Constructing  Hjelmslev Planes}
\author{Joanne L. Hall and Asha Rao}
\institute{Joanne L. Hall \at  School of Mathematical Sciences, 
Queensland University of Technology. \email{j42.hall@qut.edu.au}
\and Asha Rao \at School of Mathematical and Geospatial Sciences, 
 RMIT University. \email{asha@rmit.edu.au}}
%
%
\maketitle

\abstract{Projective Hjelmslev planes and affine Hjelmslev planes  are  generalisations of projective planes and affine planes. We present an algorithm for constructing   projective Hjelmslev planes and affine Hjelmslev planes  that uses projective planes, affine planes and  orthogonal arrays. We show that all $2$-uniform projective Hjelmslev planes, and all $2$-uniform affine Hjelmslev planes can be constructed in this way. As a corollary it is shown that all $2$-uniform affine Hjelmslev planes are sub-geometries of $2$-uniform projective Hjelmslev planes.
}
\section{Introduction}

In 1916, Hjelmslev  introduced the concept of a projective Hjelmslev  geometry as a ``\emph{geometry of reality}'' \cite{Hje16}, introducing the intriguing concept  of point and line neighbourhoods, a  property  that varied the restriction that  two points lie on exactly one line. However it was the 1950s before Klingenberg first formally defined Hjelmslev planes \cite{kling54}.  More work describing these geometries was done in the 60s and 70s by Drake and others \cite{drake70,drake74,craig64}.

Recently there has been renewed interest in these structures with Honold and Landjev \cite{HL01} showing a connection with linear codes and Saniga and Planat \cite{SP06} conjecturing that there may be connections to mutually unbiased bases. A difficulty  turns out to be explicit constructions and concrete examples, which would help further investigations. Some Hjelmslev planes can be constructed using chain rings \cite{Veld95}, those that have been the most thoroughly investigated have been the Galois ring Hjelmslev geometries \cite{HL01,KKK2011,KZ2013}, though there has been some more  recent work using other chain rings \cite{HK2013A,HK2013}.  

Note that just as  there are affine planes and projective planes which cannot be constructed via a Galois field, similarly there are Hjelmslev planes which cannot be constructed using a chain ring. Thus, having a general algorithm which generates Hjelmslev planes regardless of their algebraic structure, would be useful. In this paper we give such an algorithm for constructing projective Hjelmslev planes, using a projective plane, an affine plane and an orthogonal array. This algorithm is easily implemented using any programming language and generates a Hjelmslev plane for use in applications.  Indeed it has been implemented using Python and in the conclusion, we provide a link to a  visualisation. 

There are open online lookup tables (eg. \cite{designtheory,sloaneOA}) for the structures needed for this construction, or some may be constructed using a Galois field \cite[\S VII.2]{CRC}. 
Drake and Shult \cite[Prop 2.4]{drake76} show that all Hjelmslev planes can be constructed using a projective plane and semi-nets with zings, however there is no library of semi-nets with zings. While our algorithm \ref{alg:complete} for constructing Projective Hjelmslev planes is implied by the proof of \cite[Thm 1]{HVM1989}, we present it here in algorithmic language.  In addition the algorithm \ref{alg:affine} for constructing affine Hjelmslev planes is new.

 All affine planes are sub-geometries of projective planes, however there exist affine Hjelmslev planes which are not sub-geometries of projective Hjelmslev planes {\bf \cite{Bacon74}}.  A result of the presented algorithms is a new proof showing that all $2$-uniform affine Hjelmslev planes are sub-geometries of projective Hjelmslev planes.

 While the major drawback of this algorithm is the large number (factorial in the order of the projective plane) of isomorphic Hjelmslev planes which are generated via different paths, the algorithm is still useful in providing examples and allowing for visualization of these beautiful, but intricate structures. The authors have been as yet unable to come up with a way to address the problem of Hjelmslev isomorphisms. A 1989 paper of Hanssens and Van Maldeghm \cite{HVM1989} gives some results regarding isomorphisms of 3- and higher uniform Hjelmslev planes, but the same cannot be said for 2-uniform Hjelmslev planes which seems to be a more difficult problem.

The paper is organised in the following manner. Section \ref{def} gives definitions and facts about projective Hjelmslev planes while Section \ref{ingredients} describes the mathematical objects required for the construction.  Sections \ref{sec:algorithm} and \ref{sec:affine}   detail algorithms for constructing $2$-uniform projective Hjelmslev planes and $2$-uniform affine Hjelmslev planes.   Section \ref{fab} discusses some properties of the algorithms and  concludes with some further directions.

\section{Definitions \label{def}}

  Hjelmslev planes are generalisations of projective and affine planes, see for example \cite{Dem97,CRC}.  We write $P\in h$, to indicate that the point $P$ is incident with line $h$. A line may also be represented as the set of points incident with it.  In Hjelmslev geometry,  lines (or  points) may have the relationship of being neighbours, denoted $g\sim h$ (or $P\sim Q$).
\begin{defi}\label{PH}\cite{drake74}
 A projective Hjelmslev plane, $\mathcal{H}$, is an incidence structure such that:
\begin{enumerate}
 \item \label{points} any two points are incident with at least one line.
 \item \label{lines} any two lines intersect in at least one point.
\item \label{nlines} any two lines $g$, $h$, that intersect at more than one point are neighbours, denoted $g\sim h$.
\item \label{npoints} any two points $P$ and $Q$ that are incident with more that one line are neighbours, denoted $P\sim Q$.
\item \label{epi} there exists an epimorphism $\phi$ from $\mathcal{H}$ to an ordinary projective plane $\mathcal{P}$ such that for any points $P,Q$ and lines $g,h$ of $\mathcal{H}$:
\begin{enumerate}
\item $\phi(P)=\phi(Q)\iff P\sim Q$,
\item $\phi(g)=\phi(h) \iff g\sim h$.
\end{enumerate}
\end{enumerate}
\end{defi}
The neighbour property is an equivalence relation, so the set of lines of $\mathcal{H}$ is partitioned into line-neighbourhoods and the set of points of $\mathcal{H}$ is partitioned into point-neighbourhoods.  A projective Hjelmslev plane is denoted by $(t,r)PH$-plane where $t$ is the number of neighbouring points on each line, and $r$ is the order of the projective plane associated  by the epimorphism $\phi$.  Furthermore it is known that $r=s/t$ where $t+s$ is the number of points on each line, each line-neighbourhood has $t^2$ lines and each point-neighbourhood has  $t^2$ points \cite
{drake74,Klein59}.   Just as in projective planes, points and lines are dual.   A   $(1,r)PH$-plane is a projective plane of order $r$.  This notation should  not be confused with $PH(R)$, the projective Hjelmslev plane over the ring $R$, or $PH(n,p^r)$ where $R=GR(p^n,r)$ \cite{SP06,HL01}.

\begin{defi}\label{defi:AH}\cite{drake74} An affine Hjelmslev plane $\mathcal{H}$ is an incidence structure such that
\begin{enumerate}
\item \label{apoints} any two points are incident with at least one line.
\item \label{anlines} any two lines $g$, $h$ that intersect at more than one point are neighbours, denoted $g\sim h$.
\item \label{aepi} There exists an epimorphism $\phi$ from $\mathcal{H}$ to an ordinary affine plane $\mathcal{A}$  such that for any points $P,Q$ and lines $g,h$ of $\mathcal{H}$:
\begin{enumerate}
\item $\phi(P)=\phi(Q)\iff P\sim Q$,
\item $\phi(g)=\phi(h)\iff g\sim h$,
\item $|g\cap h|=0\Rightarrow \phi(g)\parallel \phi(h)$.
\end{enumerate}
\end{enumerate}
\end{defi}
An affine Hjelmslev plane may be denoted $(t,r)AH$-plane, where $t$ is the number of neighbouring points on each line, and $r=s/t$ where $s$ is the number of points on each line.  An affine Hjelmslev plane may have parallel lines which are neighbours as condition $3c$ of Definition \ref{defi:AH} is a one way implication.

  Hjelmslev planes are  mentioned in some books on finite geometry \cite[\S 7.2]{Dem97}, but not in the more standard works on design theory eg. \cite{CRC,BJL99}.
Hjelmslev planes have a rich structure with several interesting substructures.  The following function allows interrogation of the structure of each point neighbourhood.

\begin{defi}\cite[Defi 2.3]{drake70} Let $P$ be a point of a Hjelmslev plane $\mathcal{H}$.   The point-neighbourhood restriction, $\bar{P}$, is defined as follows:
\begin{enumerate}
\item the points of $\bar{P}$  are the points $Q$ of $\mathcal{H}$ such that $Q\sim P$.
\item the lines of $\bar{P}$ are the restrictions of lines $g$ of $\mathcal{H}$ to the points in $\bar{P}$: $g_P=g\cap \bar{P}$.
\end{enumerate}
\end{defi}

\begin{defi}\label{uniform} \cite[Defi 2.4]{drake70} A 1-uniform projective (affine) Hjelmslev plane $\mathcal{H}$ is an ordinary  projective (affine) plane.
A projective (affine) Hjelmslev plane is $n$-uniform if
\begin{enumerate}
 \item for every point $P\in\mathcal{H}$,  $\bar{P}$ is an $(n-1)$ uniform affine Hjelmslev plane.
 \item for each point $P$ of $\mathcal{H}$,   every line of $\bar{P}$ is the restriction of the same number of lines of $\mathcal{H}$.
\end{enumerate}
\end{defi}
In a $2$-uniform projective Hjelmslev plane every point-neighbourhood restriction is an ordinary affine plane.  It is known that a projective Hjelmslev plane is $2$-uniform if and only if it is a $(t,t)PH$- plane, similarly an affine Hjelmslev plane is $2$-uniform if and only if it is a $(t,t)AH$- plane \cite{Klein59}. All $(t,t)PH$- planes that can be generated by rings have been catalogued \cite{HL01}.  However there are many more Hjelmslev planes that cannot be generated by rings.

\section{\label{ingredients}The Objects Required for the Construction}
We present an algorithm for generating a $2$-uniform projective Hjelmslev plane.  This algorithm is inspired by a construction by  Drake and Shult \cite[Prop 2.4]{drake76}  and a construction by Hanssens and VanMaldeghem \cite{HVM1989}.  The construction of Drake and Schult  uses a projective plane and a semi-net with zings;  semi-nets have not been well catalogued, making Drake and Shult's construction difficult to implement. The construction of  Hanssens and Van Maldeghem is not given in algorithmic language.    The algorithms  given in section \ref{sec:algorithm} and \ref{sec:affine}  are restricted to $(t,t)PH$- planes and $(t,t)AH$- planes, but use objects that are well studied, and catalogued.
The algorithms takes three different classes of combinatorial structures and use them to generate  $(t,t)PH$- planes and $(t,t)AH$- planes.  Examples and properties of the combinatorial structures can be found in any standard work on combinatorial designs eg. \cite{CRC,BJL99}.

\begin{defi}\cite[VII.2.2]{CRC}\label{projectiveplane}A \emph{projective plane} is a set of points and lines such that
\begin{itemize}
\item
any two distinct points are incident with exactly one line.
\item
any two distinct lines intersect at exactly one point.
\item
there exist four points no three of which are on a common line.
\end{itemize}\end{defi}
A projective plane of order $m$ has $m+1$ points on each line, $m+1$ lines through each point, $m^2+m+1$ points and $m^2+m+1$ lines.  A projective plane of order  $m$ may be represented as a $2-(m^2+m+1,m+1,1)$ block design  \cite[Defi 2.9]{BJL99}.

\begin{defi}\cite[VII.2.11]{CRC}
A \emph{finite affine plane} is a set of points and lines such that
\begin{itemize}
\item
any two distinct points are incident with exactly one line.
\item
for any point $P$ not on line $l$ there is exactly one line through $P$ that is parallel (has no points in common) with $l$.
\item
there exist three points not on a common line.
\end{itemize}
\end{defi}

An affine plane of order $m$ has $m$ points on a line, $m+1$ lines through each point, $m^2$ points and $m^2+m$ lines.  An affine plane can be obtained from a projective plane by removing one line and all the points on that line.  An affine plane of order $m$ may be represented as a $2-(m^2,m,1)$ block design \cite[Defi 2.9]{BJL99}. An affine plane of order $m$  may be partitioned into $\parallel$-classes, with each $\parallel$-class containing $m$ lines.

$\parallel$-classes $S$ and $T$ are orthogonal if each line of $S$ has exactly one element in common with each line of $T$.  An affine plane of order $m$ has $m+1$ mutually orthogonal $\parallel$-classes.

\begin{defi}\cite[III.3.5]{CRC}
An \emph{orthogonal array} $OA(t,k,v)$ is a $v^2\times k$ array with entries from a set of $v$ symbols, such that in any $t$ columns each ordered $t-tuple$ occurs exactly once.
\end{defi}
Each symbol occurs in each column of the orthogonal array $v$ times. An orthogonal array may be obtained from an affine plane by assigning each point to a row, and each $\parallel$-class to a column of the array.  The symbol in position $i,j$ of the array indicates the line of $\parallel$-class $j$ that is incident with point $i$.   Alternately this structure may be considered as the dual of an affine plane.  Thus a catalogue of affine planes also provides a catalogue of the appropriate orthogonal arrays.

The three structures above may all be generated from a projective plane.  However, it is not essential in the following algorithms that the objects have any relationship other than size.
The construction of Drake and Shult \cite{drake76} uses the incidence matrix of a projective plane, for which each symbol is replaced by a sub-matrix  which meets a set of conditions given by a semi-net with zings.  The algorithm developed here is related to this method.

\section{\label{sec:algorithm} An Algorithm for Constructing  $2-$uniform Projective Hjelmslev Planes}

This algorithm is implicit in the proof of \cite[Thm 1]{HVM1989}.

\begin{alg}\label{alg:complete}
 To  create the structure $\mathcal{H}$:
\begin{enumerate}
\item
Let $\mathcal{P}$ be a projective plane of order $m$, Let  $\mathcal{A}_0,\mathcal{A}_1,\dots,\mathcal{A}_{m^2+m}$be a list of  affine planes of order $m$ and $\mathcal{O}_0,\mathcal{O}_1,\dots,\mathcal{O}_{m^2+m}$ be a list of  orthogonal arrays $OA(2,m+1,m)$.   Note that some (or all) of the affine planes and some (or all) of  orthogonal arrays may be the same.
\item
For each point of $\mathcal{P}$, replace point $P$  with  $m^2$ points which are a copy of $\mathcal{A}_{P}$. This gives $(m^2+m+1)m^2$ points in $\mathcal{H}$. Each affine plane will now be called a point-neighbourhood of $\mathcal{H}$.
\item
Choose a line $g=\{P_0,P_2,\dots ,P_m\},$ in $\mathcal{P}$,  and for each point of $g$  choose a parallel class of each  of $\mathcal{A}_{P_0},\mathcal{A}_{P_1},\dots ,\mathcal{A}_{P_m}$.
Label each of the lines of the parallel class of each point-neighbourhood with the symbols from $\mathcal{O}_g$. Since each point-neighbourhood is in $m$ lines of $\mathcal{P}$, each time a particular point $P$ of $\mathcal{P}$ is in a chosen line, a different parallel class of $\mathcal{A}_P$ must be used.  Label each column of $\mathcal{O}_g$ with a point of $g$.
\item
Each line of $\mathcal{H}$ is constructed by reading a row of $\mathcal{O}_g$ and selecting the points which correspond to the lines of the the point neighbourhoods.
\end{enumerate}
Repeat for each line of $\mathcal{P}$.
\end{alg}

An example is given  below.

Step 1.  A Projective plane of order $3$, an Affine plane of order $3$ and an orthogonal array $OA(2,4,3)$.  In this example $\mathcal{A}_i=\mathcal{A}$ and $\mathcal{O}_i=\mathcal{O}$ for all $0\leq i\leq m^2+m$.

$$\mathcal{P}=\left\{\begin{array}{l}
\{0,1,2,9\},\\
\{3,4,5,9\},\\
\{6,7,8,9\},\\
\{0,3,6,A\},\\
\{1,4,7,A\},\\
\{2,5,8,A\},\\
\{0,4,8,B\},\\
\{1,5,6,B\},\\
\{2,3,7,B\},\\
\{0,7,5,C\},\\
\{1,3,8,C\},\\
\{2,4,6,C\},\\
\{9,A,B,C\},
\end{array}\right\}
\quad\quad
\mathcal{A}=\left\{\begin{array}{l}
\{R,S,T\},\\
\{U,V,W\},\\
\{X,Y,Z\},\\
\{R,U,X\},\\
\{S,V,Y\},\\
\{T,W,Z\},\\
\{R,V,Z\},\\
\{S,W,X\},\\
\{T,U,Y\},\\
\{R,W,Y\},\\
\{S,U,Z\},\\
\{T,V,X\}
\end{array}\right\}
\quad\quad
\mathcal{O}=\begin{array}{cccc}
L & L & L & L \\
L & M & M & M \\
L & N & N & N \\
M & L & M & N \\
M & M & N & L \\
M & N & L & M \\
N & L & N & M \\
N & M & L & N \\
N & N & M & L
     \end{array}$$

Step 2. The points of $\mathcal{H}$ can be written as a Cartesian product of the set of points in $\mathcal{A}$ and $\mathcal{P}$. 
\[
 \{(0,R), (0,S), (0,T), (0,U), (0,V), (0,W), (0,X),  \dots, (C,X), (C,Y), (C,Z)\}.
\]

Step 3.  Choosing line $g=\{3,4,5,9\}$, the chosen $\parallel$-classes of each point neighbourhood of $g$, and the labels for the columns of $\mathcal{O}$.

\begin{center}In neighbourhood 3; $L:= \{ R,  S,  T\}$,
 $M := \{ U , V , W\}$,
 $N := \{ X , Y , Z\}$.

In neighbourhood 4; $L :=\{R, S, T\}$,
     $M :=\{U, V, W\}$,
     $N :=\{X, Y, Z\}$.

In neighbourhood 5;  $L :=\{R, S, T\}$,
     $M:=\{U, V, W\}$,
     $N:=\{X, Y, Z\}$.

In neighbourhood 9; $L :=\{R, U, X\}$,
     $M:=\{S, V, Y\}$,
     $N:=\{T, W, Z\}$.
\end{center}
$$\begin{array}{cccc}
3 & 4 & 5 & 9\\
\hline
L & L & L & L \\
L & M & M & M \\
L & N & N & N \\
M & L & M & N \\
M & M & N & L \\
M & N & L & M \\
N & L & N & M \\
N & M & L & N \\
N & N & M & L
     \end{array}$$

Step 4.  The lines of $\mathcal{H}$ in the line-neighbourhoods corresponding to the line $\{3,4,5,9\}$  of $\mathcal{P}$ are constructed according to $\mathcal{O}$. Note that every pair of lines from within a line neighbourhood share exactly 3 points.

\[
\begin{array}{cccccccccccc}
\{(3,R), (3,S), (3,T),&  (4,R), (4,S), (4,T),&  (5,R), (5,S), (5,T),&  (9,R), (9,U), (9,X)\}\\
\{(3,R), (3,S), (3,T),&  (4,U), (4,V), (4,W),&  (5,U), (5,V), (5,W),&  (9,S), (9,V), (9,Y)\}\\
\{(3,R), (3,S), (3,T),&  (4,X), (4,Y), (4,Z),&  (5,X), (5,Y), (5,Z),&  (9,T), (9,W), (9,Z)\}\\
\{(3,U), (3,V), (3,W),&  (4,R), (4,S), (4,T),&  (5,U), (5,V), (5,W),&  (9,T), (9,W), (9,Z)\}\\
\{(3,U), (3,V), (3,W),&  (4,U), (4,V), (4,X),&  (5,X), (5,Y), (5,Z),&  (9,R), (9,U), (9,X)\}\\
\{(3,U), (3,V), (3,W),&  (4,X), (4,Y), (4,Z),&  (5,R), (5,S), (5,T),&  (9,S), (9,V), (9,Y)\}\\
\{(3,X), (3,Y), (3,Z),&  (4,R), (4,S), (4,T),&  (5,X), (5,Y), (5,Z),&  (9,S), (9,V), (9,Y)\}\\
\{(3,X), (3,Y), (3,Z),&  (4,U), (4,V), (4,X), & (5,R), (5,S), (5,T),&  (9,T), (9,W), (9,Z)\}\\
\{(3,X), (3,Y), (3,Z),&  (4,X), (4,Y), (4,Z),&  (5,U), (5,V), (5,W),&  (9,R), (9,U), (9,X)\}\\
\end{array}
\]

\begin{thm} \label{thm:proof}  The structure generated by Algorithm \ref{alg:complete} is a 2-uniform $(m,m)PH$-plane.
\end{thm}

 \begin{proof}
This algorithm generates an incidence structure $\mathcal{H}$ with $(m^2+m+1)m^2$ points, $(m^2+m+1)m^2$ lines,  each line containing $(m^2+m)$ points, and each point incident with $(m^2+m)$ lines.

Axioms \ref{points} and \ref{npoints}:  Any pair of points $P$ and $Q$ which are in the same point-neighbourhood are incident with exactly one line of the point neighbourhood restriction, which is an affine plane.  Each line of the point neighbourhood restriction is used in $m$ lines of $\mathcal{H}$, as each symbol appears $m$ times in each column of $\mathcal{O}$.  For points $P$ and $R$ which  are in different point-neighbourhoods, there is exactly one line of $\mathcal{P}$ which is incident with any pair of point-neighbourhoods. Given $\parallel$-classes $\bar{P}_X$ and $\bar{R}_Y$ of each point-neighbourhood, $\mathcal{O}$ ensures that each line of $\bar{P}_X$ is in a line of $\mathcal{H}$ with each line of $\bar{R}_Y$ exactly once.

Axioms \ref{lines} and \ref{nlines}: $\mathcal{O}$ ensures that lines in the same line neighbourhood intersect in exactly one line of a $\parallel$-class of a point-neighbourhood, which is $m$ points.  For lines $g$ and $h$ which are in different line-neighbourhoods, their line-neighbourhoods may be labled with lines from $\mathcal{P}$.  Any pair of lines in $\mathcal{P}$ intersect in exactly one point, thus any line-neighbourhoods of $\mathcal{H}$ intersect in exactly one point neighbourhood $\bar{Q}$.  Each line-neighbourhood is allocated a different $\parallel$-class $\bar{Q}_X$, $\bar{Q}_Y$.  Thus the line $g$ in $\mathcal{H}$ must contain a line of $\bar{Q}_X$ and $h$ a line of $\bar{Q}_Y$.  As the $\parallel$-classes are orthogonal, $g$ and $h$ intersect in exactly one point.

Let $\phi$ collapse point-neighbourhoods  and line neighbourhoods.  It is trivial to check that this is incidence preserving and surjective, and thus an epimorphism.

All axioms of Definition \ref{PH} are satisfied.  Thus $\mathcal{H}$ is a projective Hjelmslev plane.

The cardinalities of $\mathcal{H}$ show that i$\mathcal{H}$ is a $(m,m)PH$-plane and is therefore $2$-uniform.
\end{proof}

In the example the affine plane used is a sub-geometry of the given projective plane.  However this is not required and any affine planes of the appropriate size may be used.  Any projective plane, any affine planes and any orthogonal arrays of the appropriate sizes may be used.

\begin{thm}
 All $2$-uniform projective Hjelmslev planes can be generated using Algorithm \ref{alg:complete}.
\end{thm}
\begin{proof}
From Theorem \ref{thm:proof},   $\mathcal{H}$ is a $2$-uniform projective Hjelmslev plane.
Axiom 1 of Definition \ref{uniform} requires that point-neighbourhood restrictions are  affine  planes.  This is guaranteed by step 2.  Requiring that every line of every $\bar{P}$ is the restriction of the same number of lines is equivalent to ensuring that each line  of each parallel class of the point-neighbourhood is  included in the same number of lines at step 4.  This is ensured as each symbol occurs in each column of an orthogonal array the same number of times.
\end{proof}

\section{An Algorithm for Constructing $2$-uniform Affine Hjelmslev Planes  \label{sec:affine}}
Algorithm \ref{alg:complete} may be amended to construct $2$-uniform affine Hjelmslev planes

\begin{alg}\label{alg:affine}
 To  create the structure $\mathcal{H}$:
\begin{enumerate}
\item
Let $\mathcal{A}$ be a affine plane of order $m$, Let  $\mathcal{A}_0,\mathcal{A}_1,\dots,\mathcal{A}_{m^2}$be a list of affine planes of order $m$ and $\mathcal{O}_0,\mathcal{O}_1,\dots,\mathcal{O}_{m^2+m}$ be a list of orthogonal arrays $OA(2,m,m)$.   Note that some (or all) of the affine planes and orthogonal arrays may be the same.
\item
For each point of $\mathcal{A}$, replace point $P$  with  $m^2$ points which are a copy of $\mathcal{A}_{P}$. This gives $m^4$ points in $\mathcal{H}$. Each affine plane will now be called a point-neighbourhood of $\mathcal{H}$.
\item
Choose a line $g=\{P_0,P_2,\dots, P_{m-1}\},$ in $\mathcal{H}$,  and for each point of $g$  choose a parallel class of each  of $\mathcal{A}_{P_0},\mathcal{A}_{P_1},\dots, \mathcal{A}_{P_{m-1}}$.
Label each of the lines of the parallel class of each point-neighbourhood with the symbols from $\mathcal{O}_g$. Since each point-neighbourhood is in $m$ lines of $\mathcal{P}$, each time a particular point $P$ of $\mathcal{P}$ is in a chosen line, a different parallel class of $\mathcal{A}_P$ must be used.  Label each column of $\mathcal{O}_g$ with a point of $g$.
\item 
Each line of $\mathcal{A}$ is constructed by reading a row of $\mathcal{O}_g$ and selecting the points which correspond to the lines of the the point neighbourhoods.
\end{enumerate}
Repeat for each line of $\mathcal{H}$.
\end{alg}

\begin{thm}\label{thm:affineproof} All $2$-uniform affine Hjelmslev planes can be generated using  Algorithm \ref{alg:affine}.
\end{thm}
The proof is similar to that of Theorem \ref{thm:proof} and is omitted.

Affine planes are sub-geometries of projective planes, the same is sometimes true for affine Hjelmslev planes.
\begin{lem}\cite{drake74}A $(t,r)PH$-plane can be truncated to a $(t,r)AH$- plane.
\end{lem}
\begin{proof}
Take a $(t,r)PH$- plane and remove all the lines of one line-neighbourhood together with all incident points.
\end{proof}

However not all affine Hjelmslev planes may be extended to a projective Hjelmslev plane \cite{Bacon74}.  As a corollary of Theorems \ref{thm:proof} and \ref{thm:affineproof} we have a new proof that $(t,t)AH$- planes are well behaved in this respect.
\begin{cor}\cite{Bacon74}
All $(t,t)AH$- planes are a sub-geometry of a $(t,t)PH$- plane.
\end{cor}
\begin{proof}
All orthogonal arrays $OA(2,m,m)$ are extendible to $OA(2,m+1,m)$ \cite[\S III Thm 3.95]{CRC}, and hence all $(t,t)AH$- planes are extendible to $(t,t,)PH$- planes.
\end{proof}

\section{\label{fab}Conclusions}
For orders where there are several possible projective planes, affine planes and orthogonal arrays, these algorithms generate many different Hjelmslev planes of the same size.  The problem of determining if two planes are isomorphic is also computational intensive.  Even with the same seeds, there are an enormous number of choices to be made in step 3, resulting in an explosion of the number of planes generated.    With no cut downs, $(m^2+m+1)!(m+1)!m!m!$ projective Hjelmslev planes can be generated using Algorithm \ref{alg:complete}.  When there is an algebraic structure on the plane, the automorphism group can be very large \cite{KKK2011}.  Thus  a large number  of the planes generated by this algorithm are isomorphic.

Further analysis of the choices made in step 3 may reduce the number of isomorphic planes generated.  Some ground breaking work is needed on automorphisms of Hjelmslev planes to reduce this to a sub factorial number.  As mentioned before there has been little research on isomorphisms of Hjelmslev planes.  There are some results concerning isomorphism classes of Hjelmslev planes of uniformity $3$ or greater \cite{HVM1989}, however isomorphisms of $2$-uniform Hjelmslev planes appears to be a more difficult problem.

 This algorithm has been implemented in Python by  Jesse Waechter-Cornwill with a visual interpretation.
See \url{http://joannelhall.com/gallery/hjelmslev/}

There is also scope for extending this algorithm to affine Hjelmslev planes of higher uniformity by using the $2$-uniform affine Hjelmslev plane generated in Algorithm \ref{alg:affine}, and the corresponding array structure as inputs in Algorithm \ref{alg:complete}.

The material presented here is the final form and has not been published elsewhere.

\begin{acknowledgement}
Thanks to Jesse Waechter-Cornwill for the coding and visualisation of  algorithm 1.  Thanks are due to Cathy Baker for highlighting reference \cite{HVM1989}.
\end{acknowledgement}

\bibliographystyle{spmpsci}
\bibliography{referencelist}

\end{document}